\documentclass{amsart}
\usepackage{local}
\usepackage{enumitem}  
\begin{document}

\title{Local Transitivity and Entanglement Obstructions for Primitive Points}

\author[C. Nguyen, A. Yagci, Y. Zhou]{Chi Nguyen, Arman Yagci, Yunchuan Zhou}

\begin{abstract}
Primitive points on the tower of modular curves \(X_1(n)\) provide a finite ``certificate set'' for detecting isolated points above a fixed \(j\)-invariant: for a non-CM elliptic curve \(E/\QQ\), \(j(E)\) arises from an isolated point on some \(X_1(N)\) if and only if one of the associated primitive points is isolated. We bound the number \(|\PE|\) of primitive points in terms of the adelic index \(I(E)\) and give criteria as well as an algorithm for uniqueness of primitive point. As an application, every Serre curve has \(|\PE|=1\); hence Serre curves do not contribute isolated \(j\)-invariants.
\end{abstract}

\maketitle

\section{Introduction}
For an integer $N \geq 1$, non-cuspidal points on the modular curve \(X_1(N)\) over \(\mathbb{Q}\) parametrize isomorphism classes of elliptic curves together with a distinguished point of exact order \(N\). A central theme in the arithmetic of modular curves is to understand atypical algebraic points: points which do not lie in the expected infinite families of points of the same degree. The modern framework is organized around \textbf{parameterized points} and \textbf{isolated points}, as discussed in detail in \cite{viray2025isolatedparameterizedpointscurves}. Concretely, a closed point $x \in C$ of degree $d$ on a curve $C/k$ is isolated if it is neither explained by a degree-$d$ map $C \to \PP^1$ ($\PP^1$-parameterization) nor by a positive-rank subvariety in $\text{Jac}(C)$ (AV-parameterization) (\cite{viray2025isolatedparameterizedpointscurves}, \cite{BOURDON2019106824},\cite{Bourdon_2024}). On modular curves, isolated and sporadic points interact tightly with uniformity problems for Galois representations attached to elliptic curves: exceptional behavior of mod-$\ell$ Galois representations can force low-degree points on $X_1(\ell^r)$, which constrain the possible Galois images.

Serre's Open Image Theorem (\cite{serre1972proprietes}) implies that for a fixed non-CM elliptic curve $E/\QQ$, the mod $p$ representations $\rho_{E,p}$ are surjective for all sufficiently large primes $p$. Serre's Uniformity Conjecture (\cite{serre1972proprietes}), now a conjecture of Sutherland (\cite{SUTHERLAND_2016}) and Zywina (\cite{zywina2015possibleimagesmodell}), asks for a single bound that works simultaneously for all non-CM $E/\QQ$, often conjectured to be $37$ (\cite{zywina2024explicitopenimageselliptic}, \cite{PMIHES_1981__54__123_0}). The bridge to modular curves is that any failure of surjectivity at a prime $p$ forces the image into a proper subgroup $H\subseteq \GL_2(\FF_p)$, and hence produces a $\QQ$-rational point on the corresponding modular curve $X_H \to X(1)$. In particular, Borel-type images correspond to rational $p$-isogenies while the remaining large $p$ possibilities are controlled by Cartan-normalizer type curves (e.g. $X^+_{\text{ns}}(p)$) with $p>37$. Thus, proving uniformity becomes a problem on a (finite, effectively describable) collection of modular curves: show that for $p$ beyond a universal threshold, every $\QQ$-rational point on each relevant $X_H$ is cuspidal or CM, so no non-CM elliptic curve can realize a non-surjective $\rho_{E,p}$.

As shown in \cite{bourdon2021sporadicpointsodddegree}, if there are only finitely many isolated $j$-invariants associated to non-CM $\QQ$-curves, which are elliptic curves isogenous to their Galois conjugates, then Serre's Uniformity Conjecture holds. This motivates the study of isolated points as an effort towards answering Serre's question.

Recently, \cite{Bourdon_2024} introduced, for each non-CM elliptic curve $E/\QQ$, a finite set of points on the modular tower $\sqcup_{n\geq 1} X_1(n)$, called the primitive points associated to $E$, and denoted $\PE$. Essentially, this is the minimal set from which we can recover the degree of every closed point $x$ above $j(E)$ on some $X_1(n)$. Additionally, every closed point on some $X_1(n)$ lying above $j(E)$ maps to a unique primitive point, and $j(E)$ arises from an isolated point on some $X_1(n)$ if and only if at least one point of $\PE$ is isolated. Thus, $\PE$ is a compact ``certificate'' set capturing all levels at which $E$ can contribute isolated behavior. For example, one might ask the question: does the probability of $j(E)$ being isolated increase with larger $\lvert \PE \rvert$? On the other hand, we know that if $\lvert \PE \rvert=1$, $j(E)$ is necessarily non-isolated. Hence, this paper studies the size and structure of \(\PE\) from two complementary perspectives: uniform bound in terms of adelic invariants, and criteria for the extremal case $\lvert \PE \rvert=1$.

\subsection{Bounds from Adelic Invariants}
Let \(\rho_E:\Gal(\overline{\mathbb{Q}}/\mathbb{Q})\to \GL_2(\widehat{\mathbb{Z}})\) be the adelic representation and let \(I(E) \coloneqq[\GL_2(\widehat{\mathbb{Z}}):\rho_E(\Gal(\overline{\mathbb{Q}}/\mathbb{Q}))]\) be its index. Let \(m_0(E)\) denote the level of the $m$-adic Galois representation, where $m$ is the product of $2,3$, and primes $\ell$ where the $\ell$-adic Galois representation is not surjective. The main result of section 3 establishes the following bound for \(|\PE|\) in terms of $m_0(E)$ and $I(E)$, which may better be viewed as two separate bounds.

\begin{theorem}
Let $E/\mathbb Q$ be a non-CM elliptic curve, and let $m_0=m_0(E)$ be as above. Then
\[
|\mathcal{P}(E)|\le \min\left\{ m_0^2,\ 1+\frac{I(E)\,\sigma_0(m_0)}{2}\right\}~,
\]
where $\sigma_0(m_0)$ denotes the number of positive divisors of $m_0$.
\end{theorem}
\subsection{Uniqueness via Transitivity and Entanglement Obstruction}
Write $\rho_{E,n}$ for the mod-$n$ representation, and set
\[
H(n)=\langle \rho_{E,n}(\Gal(\overline{\QQ}/\QQ)),-I\rangle\subseteq \GL_2(\ZZ/n\ZZ).
\]
Via the standard identification $E[n]\cong (\ZZ/n\ZZ)^2$ after choosing a basis, let $V_n$ denote the subset of vectors of exact order $n$.

\begin{theorem}[Uniqueness criterion]\label{thm:intro-uniq}
Let $E/\QQ$ be non-CM and let $m_0=m_0(E)$.  Then $|\PE|=1$ if and only if, for every divisor $n\mid m_0$, the group $H(n)$ acts transitively on $V_n$.
In this case, the unique primitive point is the degree-$1$ point above $j(E)$ on $X_1(1)$.
\end{theorem}

Using this, we can reduce global transitivity on all such $n$ to two verifiable hypotheses: a \textbf{local transitivity} condition at each prime power dividing $m_0$, and a \textbf{stabilizer-surjectivity} condition that controls entanglement across coprime factors.  The latter is motivated by the fiber-product description of division-field Galois groups and the fact that entanglement fields obstruct naive product decompositions.

For coprime $a,b$ with $ab\mid m_0(E)$, let $K_a=\QQ(E[a])$, $K_b=\QQ(E[b])$, and $L_{a,b}=K_a\cap K_b$. 

\begin{theorem}[Sufficient criterion for uniqueness]\label{thm:intro-sufficient}
Let $E/\QQ$ be non-CM and let $m_0=m_0(E)$.  Suppose:
\begin{enumerate}
\item (Local transitivity) for every prime power $\ell^k\parallel m_0$, the group $H(\ell^k)$ acts transitively on $V_{\ell^k}$; and
\item (Stabilizer-surjectivity) for every coprime $a,b$ with $ab \mid m_0,$ let $L_{a,b} \coloneq K_a \cap K_b$ and $Q_{a,b} \coloneq \Gal(L_{a,b}/\QQ)$. Under restriction, the natural surjections $$\Gal(K_a/\QQ) \twoheadrightarrow Q_{a,b} \quad \Gal(K_b/\QQ) \twoheadrightarrow Q_{a,b}$$ are surjective stabilizers of exact-order points.
\end{enumerate}
Then $H(n)$ acts transitively on $V_n$ for every $n\mid m_0$, hence $|\PE|=1$.
\end{theorem}

We package these conditions into an algorithm that effectively verifies uniqueness of primitive point given the generators of the adelic Galois representation image and $m_0$. These results fit into (and are motivated by) the broader isolated-point program on modular curves.  The point of primitive points is that they isolate the genuinely “new” levels where an elliptic curve can contribute isolated behavior, and Theorems~\ref{thm:intro-uniq}--\ref{thm:intro-sufficient} show that the obstruction to uniqueness decomposes into two explicit sources: local orbit failures and entanglement phenomena.
\subsection{Serre curves and a density-one corollary}
Motivated by \cite{Jones2009-ob}'s result, where almost all elliptic curves, when organized by naive height, are Serre curves, which are non-CM $E/\QQ$ with $I(E)=2$, we pose the question: what is the characteristic of primitive points associated with Serre curves? We answer this question with the following theorem.

\begin{theorem}
If $E/\QQ$ is a Serre curve, then $|\PE|=1$.  In particular, $j(E)$ does not arise from an isolated point on any $X_1(n)$.
\end{theorem}

Combined with Jones’s theorem that Serre curves have natural density $1$ when elliptic curves over $\QQ$ are ordered by height \cite{Jones2009-ob}, we note that
when ordered by height, almost all non-CM elliptic curves $E/\QQ$ satisfy $|\PE|=1$.  Equivalently, the presence of multiple primitive points is an exceptional phenomenon.

In particular, multiple primitive points, and hence the possibility of isolated behavior in the modular tower, can be traced to explicit failures of local transitivity and/or explicit entanglement constraints.

\subsection{Outline} The paper is organized as follows. Section~2 recalls the necessary background. Section~3 proves the general bounds for \(|\PE|\). Section~4 develops the transitivity characterizations of uniqueness and proves the sufficient criterion based on local transitivity and stabilizer-surjectivity. We also give an algorithm that checks for this criterion given an elliptic curve and analyze the results. Section~5 treats Serre curves and derives the density-one corollary.

\subsection{Code} The algorithms mentioned in this paper have been implemented in SageMath \cite{sagemath}. Associated codes and data are available in the GitHub repository at \url{https://github.com/ckn2000/local-transitivity-and-entanglement}.

\subsection*{Acknowledgments} We are grateful to Abbey Bourdon for introducing the questions that led to this work, and to Sarah Arpin for her encouragement and insightful guidance throughout. This collaboration originated during the Arithmetic Geometry at UNT workshop in May 2025, and we thank Lea Beneish for organizing the event and creating the setting in which this project could begin. We also thank David Zywina for making his algorithm publicly available and David Roe for his modifications of said algorithm.

\section{Background}
\subsection{Closed Points on the Modular Curve $X_1(N)$}
Set $\GalQ:=\Gal(\overline{\QQ}/\QQ)$. The modular curve $X=X_1(N)$ can be viewed as a scheme over $\mathbb Q$. We say that a (scheme theoretic) point $x\in X$ is \textbf{closed} if $\{x\}$ is Zariski closed. The residue field $\mathbb Q(x)$ is a finite extension of $\mathbb Q$ if and only if $x$ is closed, in which case the degree of this extension is the \textbf{degree of the closed point} $x$. Closed points of $X$ are in bijection with $\GalQ$-orbits in $X(\overline{\mathbb Q})$, and the degree of a closed point is the size of the corresponding Galois orbit.

Let $E/\mathbb Q$ be a non-CM elliptic curve, and let $P=(x_0,y_0)\in E$ be a point of order $N$. Under the moduli interpretation, the pair $(E,P)$ corresponds to a non-cuspidal point $x\in X_1(N)$. If $x$ is closed, we also write $[E,P]$ to denote the closed point, and we have $\mathbb Q(x) \cong \mathbb Q(x_0)$ so that the degree of $x=[E,P]$ can be more easily computed as $[\mathbb Q(x_0):\mathbb Q]$.

\subsection{Natural Map Between Modular Curves}
For positive integers $a$ and $b$, there is a natural surjective $\mathbb Q$-rational map $f:X_1(ab)\rightarrow X_1(a)$ which sends $[E,P]\mapsto[E,bP]$ and satisfies $\deg(x)\le \deg(f)\cdot \deg(f(x))$ for closed points $x\in X_1(ab)$. The degree of $f$ is given by
\[\deg(f)=c_f \cdot b^2\prod_{p\mid b, p\nmid a}\left(1-\frac{1}{p^2}\right)~,\]
where $c_f=\frac{1}{2}$ if $a\leq 2$ and $ab>2$, and $c_f=1$ otherwise.

For natural maps $f:X_1(n)\rightarrow X_1(b)$ and $g:X_1(b)\rightarrow X_1(a)$, the composition $g\circ f$ coincides with the natural map $h:X_1(n)\rightarrow X_1(a)$ and we have $\deg(h)=\deg(f)\cdot \deg(g)$.

\begin{subsection}{Isolated Points and Primitive Points}
Specializing on modular curves $X_1(N)$, we give a definition of isolated points.
\begin{definition} Fix some modular curve $X=X_1(N)$ and let $x\in X$ be a closed point of degree $d$. Let $\phi_d:\text{Sym}^d(X)\rightarrow \text{Pic}^d(X)\cong \text{Pic}^0(X)$ be the degree $d$ Abel-Jacobi morphism.
\begin{enumerate}
\item We say $x$ is \textbf{$\mathbb P^1$-parameterized} if there exists a $\mathbb Q$-rational $x'\in \text{Sym}^d(X)$ such that $x\neq x'$ and $\phi_d(x)=\phi_d(x')$.
\item We say $x$ is \textbf{AV-parameterized} if there exists a positive rank abelian subvariety $A\subseteq \text{Pic}^0(X)$ such that $\phi_d(x)+A\subseteq \text{im}(\phi_d)$.
\item We say $x$ is \textbf{isolated} if it is neither $\mathbb P^1$-parameterized nor AV-parameterized.
\end{enumerate}
\end{definition}

For a more in-depth treatment of isolated points defined on a larger class of curves, see \cite{viray2025isolatedparameterizedpointscurves} or \cite{BOURDON2019106824}.
Next, we define the closely related primitive points.

\begin{definition}
Let $E/\mathbb Q$ be a non-CM elliptic curve. The set of \textbf{primitive points} $\mathcal P(E)$ associated to $E$ is defined as the set of all closed points $x\in \bigcup_{n\in \mathbb Z^+} X_1(n)$ with $j(x)=j(E)$ such that $\deg(x)<\deg(f)\cdot \deg(f(x))$ for all natural maps $f:X_1(n)\rightarrow X_1(n')$ taking $x\in X_1(n)$ to $f(x)\in X_1(n')$ where $n'\mid n$ is some proper divisor.
\end{definition}

We note that $\mathcal P(E)$ is always non-empty as the unique closed point in $X_1(1)$ above $j(E)$ is vacuously primitive due to the absence of a natural map to a lower level. 

The motivation for this definition is captured by Theorem 24 in \cite{Bourdon_2024}, which characterizes primitive points as a finite ``certificate set'' for detecting isolated points. By the proof of said theorem, each closed point $x\in X_1(n)$ with $j(x)=j(E)$ corresponds to a unique primitive point, namely the image of $x$ under the natural map $f:X_1(n)\rightarrow X_1(a)$ where $a\mid n$ is minimal such that $\deg(x) = \deg(f)\cdot \deg(f(x))$. Further, Remark 26 and Corollary 32 in \cite{Bourdon_2024} together show that a natural map $g:X_1(n)\rightarrow X_1(b)$ also satisfies this degree condition if and only if $a\mid b\mid n$. In this case, the primitive point corresponding to $g(x)\in X_1(b)$ is also $f(x)$.

\begin{ex}
Let $E/\mathbb Q$ be a non-CM elliptic curve and suppose $E$ has a $\mathbb Q$-rational point $P=(x_0,y_0)$ of order $n>1$. Since $x_0\in \mathbb Q$, the degree of the corresponding closed point $x=[E,P]$ is $[\mathbb Q(x_0):\mathbb Q]=1$.

Write $n=ab$ for positive integers $a,b$ with $b>1$ and note the natural map $f:X_1(ab)\rightarrow X_1(a)$ satisfies
$$\deg(f)=c_{f}\cdot b^2 \prod_{p\mid b,~p\nmid a} \left(1-\frac{1}{p^2}\right) \geq c_{f}\cdot b^2  \prod_{p\mid b} \left(1-\frac{1}{p^2}\right)\geq \frac{1}{2} \cdot \prod_{p\mid b} (p^2-1) > 1~.$$
Hence, for all proper divisors $a$ of $n$ we have \[\deg(x)=1< \deg(f)\cdot \deg(f(x))~,\]
showing that $x$ is primitive.
\end{ex}
\end{subsection}

\begin{subsection}{Galois Representations}
Let $E/\mathbb{Q}$ be a non-CM elliptic curve and write $\Gal_{\mathbb{Q}}\coloneq\Gal(\overline{\mathbb{Q}}/\mathbb{Q})$.
The Galois action on the full torsion subgroup $E(\overline{\mathbb{Q}})_{\mathrm{tors}}$ is encoded by the \textbf{adelic Galois representation}
\[
\rho_E:\Gal_{\mathbb{Q}}\rightarrow \text{Aut}(E(\overline{\mathbb Q})_{\text{tors}}) \cong \GL_2(\widehat{\mathbb{Z}}) \cong \prod_{p \text{ prime}} \GL_2(\mathbb Z_p)~.
\]
For each integer $n\ge 1$, projecting to primes dividing $n$ yields the \textbf{$n$-adic Galois representation}
\[\rho_{E,n^\infty}: \Gal_{\mathbb Q} \rightarrow \prod_{p \mid n} \GL_2(\mathbb Z_p)~,\]
which describes the action of $\GalQ$ on points whose order is divisible only by primes dividing $n$, whereas reduction modulo $n$ yields the \textbf{mod $n$ Galois representation}, which records the Galois action on points of order dividing $n$,
\[
\rho_{E,n}:\Gal_{\mathbb{Q}}\rightarrow \GL_2(\mathbb{Z}/n\mathbb{Z})~.
\]
Set 
\(
G \coloneq \rho_E(\Gal_{\mathbb{Q}})\subseteq \GL_2(\widehat{\mathbb{Z}})~,
\quad
I(E)\coloneq[\GL_2(\widehat{\mathbb{Z}}):G]~, \quad G(n)\coloneq\rho_{E,n}(\Gal_{\mathbb{Q}})~.\)
Equivalently, $G(n)$ is the image of $G$ under the natural reduction map
$\pi_n:\GL_2(\widehat{\mathbb{Z}})\twoheadrightarrow \GL_2(\mathbb{Z}/n\mathbb{Z})$. \\

Serre's Open Image Theorem says that $G$ is open in $\GL_2(\Zhat)$ and $I(E)$ is finite (\cite{serre1972proprietes}). In particular, there exists a positive integer $N$ such that $G=\pi_N^{-1}(G(N))$. The smallest such $N$ is called the \textbf{level of $G$}. Similarly, the smallest $m_0\in \mathbb Z^+$ such that $\text{im}(\rho_{E,m^\infty})=\pi_{m_0}^{-1}(G(m_0))$ is the \textbf{level of} $\text{im}(\rho_{E,m^\infty})$.
Let
\[
S_E\coloneq\{2,3\} \cup \{\ell \text{ prime}:\rho_{E,\ell^\infty}\ \text{is not surjective}\}~,
\qquad
m\coloneq\prod_{\ell\in S_E}\ell~,
\]
and let $m_0$ denote the level of the $m$-adic image. Then, by Proposition 28 in \cite{Bourdon_2024}, every primitive point attached to $E$ occurs on $X_1(n)$ for some divisor $n\mid m_0$.
\end{subsection}

\section{Bound on the Number of Primitive Points}
Let $E/\QQ$ be a non-CM elliptic curve and let $E[n]$ denote the $n$-torsion points on $E$. Fix $n\ge 1$. The noncuspidal geometric points of $X_1(n)$ lying above $j(E)$ correspond to isomorphism classes of pairs $(E,\{\pm P\})$ with $P\in E[n]$ of exact order $n$.  Choosing a basis identifies the free $(\mathbb{Z}/n\mathbb{Z})$-module $E[n]\cong (\mathbb{Z}/n\mathbb{Z})^2$, and under this identification, points $P\in E[n]$ of exact order $n$ correspond precisely to vectors in $V_n:$
\[
V_n\coloneq\{v=(a,b) \in (\ZZ/n\ZZ)^2: \ord(v)=n\}~,
\]
where $v=(a,b)=cP+bQ$ for $(P,Q)$ fixed chosen basis of $E[n]$ and $\ord(v)$ denotes the additive order of $v$ in the abelian group $(\ZZ/n\ZZ)^2$.

\begin{lemma}
If $v=(a,b)\in (\ZZ/n\ZZ)^2$, then
\[
\ord(v)=\frac{n}{\gcd(a,b,n)}~.
\]
In particular, $v\in V_n$ if and only if $\gcd(a,b,n)=1$.
\end{lemma}

\begin{proof}
Note $\overline{d}\coloneq\overline{\ord(v)} \in \ZZ/n\ZZ$ generates the annihilator ideal
$\text{Ann}(v)=d\ZZ/n\ZZ \leq \ZZ/n\ZZ$.
The conditions $d\mid n$ and $da\equiv db\equiv 0 \pmod{n}$ together with the minimality of $d$
then force $d=\frac{n}{\gcd(a,b,n)}$.
The final assertion follows.
\end{proof}

\begin{lemma}\label{point-orbit bijection}
The closed points $x\in X_1(n)$ satisfying $j(x)=j(E)$ are in bijection with the $H(n)$-orbits on $V_n$, where \(H(n)\coloneq\langle G(n),-I\rangle \subseteq \GL_2(\ZZ/n\ZZ).
\)
\end{lemma}

\begin{proof}
Closed points of $X_1(n)$ are $\Gal_{\QQ}$-orbits of geometric points.
A geometric point of $X_1(n)$ above $j(E)$ is represented by a pair $(E,P)$
with $P\in E[n]$ of exact order $n$, modulo automorphisms of $(E,P)$.

For $\sigma\in \Gal_{\QQ}$ one has $\sigma\cdot(E,P)=(E,\sigma(P))$,
and relative to a chosen basis of $E[n]$ the action of $\sigma$ is that of
$\rho_{E,n}(\sigma)\in G(n)$.

Since $E$ is non-CM, we have
$\text{Aut}(E_{\overline{\QQ}})=\{\pm 1\}$, and the automorphism $-1$ sends $(E,P)$ to $(E,-P)$,
which corresponds to the action of $-I$ on $(\ZZ/n\ZZ)^2$.
Therefore closed points of $X_1(n)$ above $j(E)$ correspond to orbits of
$\langle G(n),-I\rangle=H(n)$ acting on $V_n$.
\end{proof}

\begin{theorem}
The natural action of $\GL_2(\ZZ/n\ZZ)$ on $(\ZZ/n\ZZ)^2$ is transitive on $V_n$.
\end{theorem}

\begin{proof}
Let $v=(a,b)\in V_n$, so $\gcd(a,b,n)=1$ by the previous lemma.
Hence we may choose $u,w\in \ZZ/n\ZZ$ with $ua+wb\equiv 1\pmod n$ and set
\[
M\coloneq\begin{pmatrix} a & -w \\ b & u \end{pmatrix}.
\]
Then $\det(M)=au+bw\equiv 1\pmod n$, hence $M\in \GL_2(\ZZ/n\ZZ)$, and $Me_1=v$ for $e_1=(1,0)$.
This shows that every $v\in V_n$ lies in the orbit of $e_1$.
\end{proof}

\begin{remark} \label{transitivity of SL_2}
In the proof of the above theorem, we in fact have $M\in \SL_2(\ZZ/n\ZZ)$,
so $\SL_2(\ZZ/n\ZZ)$ acts transitively on $V_n$.
\end{remark}

\begin{corollary}\label{bound on orbits}
Let $r_n$ denote the number of $H(n)$-orbits on $V_n$. Then
\[
r_n\le [\GL_2(\ZZ/n\ZZ):H(n)]~.
\]
\end{corollary}

\begin{proof}
Since $V_n$ is a single $\GL_2(\ZZ/n\ZZ)$-orbit and $H(n)$ is a subgroup of $\GL_2(\ZZ/n\ZZ)$,
the orbit decomposition under $H(n)$ has at most $[\GL_2(\ZZ/n\ZZ):H(n)]$ components.
\end{proof}

\begin{lemma}\label{bound with index}
Let $\pi_n:\GL_2(\Zhat)\twoheadrightarrow \GL_2(\ZZ/n\ZZ)$ be reduction modulo $n$
and set $K(n)\coloneq\ker(\pi_n)$.
Let $G\coloneq\rho_E(\Gal_{\QQ})\subseteq \GL_2(\Zhat)$ and $\widetilde G\coloneq\langle G,-I\rangle$.
Then
\[
[\GL_2(\ZZ/n\ZZ):H(n)]
=
[\GL_2(\Zhat):\widetilde G\,K(n)]
\le
[\GL_2(\Zhat):G]=I(E)~.
\]
\end{lemma}

\begin{proof}
The subgroup $K(n)$ is normal in $\GL_2(\Zhat)$ and $\widetilde G\,K(n)$ is a subgroup containing $\widetilde G$.
The restriction of $\pi_n$ to $\widetilde G\,K(n)$ induces a surjection
$\widetilde G\,K(n)\twoheadrightarrow H(n)$ with kernel $K(n)$, hence
$\widetilde G\,K(n)/K(n)\cong H(n)$.

Since $\GL_2(\Zhat)/K(n)\cong \GL_2(\ZZ/n\ZZ)$, the index identity follows.
The inequality is immediate from $G\subseteq \widetilde G\subseteq \widetilde G\,K(n)$.
\end{proof}

Combining Corollary~\ref{bound on orbits} and Lemma~\ref{bound with index} gives
\[
r_n\le I(E)\qquad\text{for all }n\ge 1.
\]

In fact, we will see that $r_n \le I(E)/2$ for $n<m_0$. For that we need the following lemma.

\begin{lemma}\label{preimage-criterion}
For each $n>0$, let $G(n)=\pi_n(G)$ and $K(n)=\ker(\pi_n)$ as above.
Then $\pi_n^{-1}(G(n)) = G$ if and only if $K(n)\subseteq G$.
\end{lemma}

\begin{proof}
Suppose $\pi_n^{-1}(G(n)) \subseteq G$ and let $g\in K(n)$.
Then $\pi_n(g)=1\in G(n)$, which implies $g\in \pi_n^{-1}(G(n))\subseteq G$.

Conversely, suppose $K(n)\subseteq G$ and let $g\in \pi_n^{-1}(G(n))$.
Then $\pi_n(g)\in\pi_n(G)$, meaning there exists $g'\in G$ such that $\pi_n(g)=\pi_n(g')$.
As $\pi_n$ is a group homomorphism, we have $\pi_n(gg'^{-1})=1$, hence $gg'^{-1}\in K(n)\subseteq G$,
so $g\in Gg'=G$. Thus $\pi_n^{-1}(G(n)) \subseteq G$, and the reverse inclusion is immediate.
\end{proof}

Let $N$ be the level of $G$, and note $m_0\mid N$ by Proposition~22 in \cite{Bourdon_2024}. By definition, $N$ is the smallest positive integer such that $\pi_N^{-1}(G(N))=G$.
By Lemma~\ref{preimage-criterion}, this implies that for each $n<m_0\le N$ we have $K(n)\not \subseteq G$.
Hence $G\subsetneq G\,K(n)$ so that $[G\,K(n):G]\ge 2$.

Now $I(E)=[\GL_2(\Zhat):G] = [\GL_2(\Zhat):G\,K(n)]\cdot [G\,K(n):G]$,
so for $n<m_0$ we have
\[
[\GL_2(\Zhat):G\,K(n)] = \frac{I(E)}{[G\,K(n):G]}\le \frac{I(E)}{2}~.
\]
Since $\widetilde G\,K(n)\supseteq G\,K(n)$, it follows that
\[
[\GL_2(\Zhat):\widetilde G\,K(n)]
\le
[\GL_2(\Zhat):G\,K(n)]
\le \frac{I(E)}{2}~.
\]
Using Corollary~\ref{bound on orbits} and Lemma~\ref{bound with index}, we conclude that
\[
r_n\le [\GL_2(\ZZ/n\ZZ):H(n)]
=
[\GL_2(\Zhat):\widetilde G\,K(n)]
\le \frac{I(E)}{2}
\qquad(n<m_0).
\]
We collect our findings in the following corollary:
\begin{corollary}\label{half-index cor} Let $E/\mathbb Q$ be a non-CM elliptic curve, and let $m_0$ and $r_n$ (for $n\mid m_0$) be as above. Then,
\[
r_n\le \frac{I(E)}{2}
\quad(n<m_0)~, \qquad r_{m_0} \leq I(E)~.
\]
\end{corollary}

From our discussion, note we also have $r_{m_0}\leq I(E)/2$ provided $m_0\neq N$. We can now state the main theorem of the section.

\begin{theorem} \label{bounds}
Let $E/\mathbb Q$ be a non-CM elliptic curve, and let $m_0$ be as above. Then
\[
|\mathcal{P}(E)|\le \sum_{n\mid m_0}r_n \leq  \min\left\{ m_0^2,\ 1+\frac{I(E)\,\sigma_0(m_0)}{2}\right\}~,
\]
where $\sigma_0(m_0)$ denotes the number of positive divisors of $m_0$
and $r_n$ denotes the number of $H(n)$-orbits on $V_n$.
\end{theorem}

\begin{proof}
By \cite{Bourdon_2024}, every primitive point for $E$ lies on $X_1(n)$ for some $n\mid m_0$.
For each such $n$, the primitive points lying on $X_1(n)$ form a subset of the closed points above $j(E)$,
and the latter are counted by $r_n$.
Summing over all $n\mid m_0$ gives $|\mathcal{P}(E)|\le \sum_{n\mid m_0} r_n$.

For the second inequality, we use Corollary \ref{half-index cor} and the fact that there is only one closed point in $X_1(1)$ above $j(E)$ to first obtain
\[
\sum_{n\mid m_0} r_n
=
r_1+r_{m_0}+ \sum_{\substack{n\mid m_0 \\ 1<n<m_0}} r_n
\le
1+I(E)+ \frac{I(E)}{2}\, (\sigma_0(m_0)-2)
=
1+\frac{I(E)\,\sigma_0(m_0)}{2}.
\]

One also has the crude bound
\[
\sum_{n\mid m_0} r_n\le \sum_{n\mid m_0} |V_n|
=
\sum_{n \mid m_0} n^2 \prod_{p \mid n} \left( 1- \dfrac{1}{p^2}\right)=m_0^2~,
\]
since for each $n$ the number of $H(n)$-orbits on $V_n$ is at most $|V_n|$. Note counting $|V_n|$ is a simple application of the inclusion-exclusion principle, while the final equality can be shown by strong induction on $m_0$. The result follows.
\end{proof}

\begin{ex} If $I(E)=2,$ then  $E$ is called a Serre curve. By \cite{serre1972proprietes}, Serre curves form a subclass of non-CM elliptic curves
for which the $\ell$-adic Galois representation is surjective for every prime $\ell>3$, and the remaining contribution
at $\ell\in\{2,3\}$ is constrained so that the induced $6$-adic level divides $24$.
In the notation of this paper, this implies $m_0=m_0(E)\mid 24$ for every Serre curve $E$. In the special case $m_0=1$, the characterization of primitive points immediately yields $|\mathcal{P}(E)|=1$, namely the unique closed point above $j(E)$ on $X_1(1)$.

Table~\ref{tab:serre-m0-bounds} lists the admissible values of $m_0$ and the corresponding numerical values of the
two general upper bounds derived above:
\[ \sum_{n\mid m_0} |V_n|=m_0^2
\qquad\text{and}\qquad
1+\frac{I(E)\,\sigma_0(m_0)}{2}~.
\] Since $E$ is Serre, $I(E)=[\GL_2(\Zhat):\rho_E(\Gal_{\QQ})]=2$,
so the index-based bound specializes to $1+\sigma_0(m_0)$.
\begin{table}[h]
\centering
\begin{tabular}{c|c|c}
$m_0$ &
$m_0^2$ &
$1+(I(E)\sigma_0(m_0))/2$ \\
\hline
$1$  & $1$   & $2$ \\
$2$  & $4$   & $3$ \\
$3$  & $9$   & $3$ \\
$4$  & $16$  & $4$ \\
$6$  & $36$  & $5$ \\
$8$  & $64$  & $5$ \\
$12$ & $144$ & $7$ \\
$24$ & $576$ & $9$ \\
\end{tabular}
\caption{Bounds for $|\mathcal{P}(E)|$ for Serre curves, as a function of $m_0\mid 24$.}
\label{tab:serre-m0-bounds}
\end{table}
\end{ex}
Except in the case $m_0=1$, the index-based bound is strictly sharper than the bound obtained from $\sum_{n\mid m_0}|V_n| = m_0^2$. In Section~5 we will show that $m_0=1$ occurs only for Serre curves. We will also prove that, in fact, a Serre curve has exactly one primitive point.

\begin{remark}
One may stop at the sharper bound
\[
\lvert \PE \rvert \leq \sum_{n \mid m_0} r_n \leq \sum_{n \mid m_0} \min\{|V_n|,\ [\GL_2(\Zn):H(n)]\}~,
\]
which records the contribution from each level $n\mid m_0$ separately. The drawback is that the right-hand side is not governed by a single numerical invariant of $E$: it depends on the entire collection of indices $\{[\GL_2(\Zn):H(n)]\}_{n\mid m_0}$, so any uniform estimate must control these indices simultaneously.

By contrast, the adelic index
\[
I(E)=[\GL_2(\Zhat):\rho_E(\GalQ)]
\]
is a single invariant that dominates every finite-level index. For each $n$, reduction $\GL_2(\Zhat)\twoheadrightarrow \GL_2(\Zn)$ gives
\[
[\GL_2(\Zn):H(n)] \leq [\GL_2(\Zn):\rho_{E,n}(\GalQ)] \leq I(E)~,
\]
so one can replace level-by-level subgroup data by the single scalar $I(E)$ (or any a priori upper bound for it), which is available in effective form in \cite{furio2025effectiveboundsadelicgalois}. Moreover, there is a conjectural finite list of possible adelic images (hence possible values of $I(E)$) for elliptic curves over $\QQ$ in \cite{zywina2024explicitopenimageselliptic}; ranging over that list yields a finite menu of uniform bounds for $\lvert \PE\rvert$.
\end{remark}

\section{Criteria for Uniqueness}
We have already shown that every primitive point lying above $j(E)$ occurs at some level $n\mid m_0$, where \(m_0=m_0(E)\) is the level of the \(m\)-adic Galois representation associated to $E$. This naturally leads to the problem of characterizing those curves for which the primitive-point set $\PE$ consists of a single element. In this section we address this question in two stages. First, we formulate several general conditions, each equivalent to the uniqueness of the primitive point, that clarify the structural mechanisms by which multiple primitive points can arise. Second, we give a concrete sufficient criterion in terms of (i) local transitivity properties of the images \(G(n)\) of the mod \(n\) Galois representations and (ii) the absence of entanglement among the relevant division fields. We conclude with examples illustrating the effectiveness of this criterion in practice.
\subsection{General Conditions}
For positive integers $n\ge 1$, let $G(n), H(n), V_n$ be as defined in previous sections. To prepare for the first criterion, we start with a lemma.

\begin{lemma}\label{some equivalent statements}
Let $E/\QQ$ be a non-CM elliptic curve and let $n>1$ be a divisor of $m_0=m_0(E)$. The following are equivalent:
\begin{enumerate}
    \item The natural map $f_n:X_1(n)\rightarrow X_1(1)$ satisfies $\deg(x)=\deg(f_n)$ for some closed point $x\in X_1(n)$ with $j(x)=j(E)$.
    \item $H(n)$ acts transitively on $V_n$.
    \item There is a unique closed point $x\in X_1(n)$ with $j(x)=j(E)$.
\end{enumerate}
\end{lemma}
\begin{proof}
The equivalence of $(2)$ and $(3)$ follows directly from Lemma \ref{point-orbit bijection}. To see the equivalence of $(1)$ and $(2)$, fix a divisor $n\mid m_0$ with $n>1$, and let $x\in X_1(n)$ be a closed point with $j(x)=j(E)$. The natural projection map $f_n:X_1(n)\rightarrow X_1(1)$ has degree given by
\[
\deg(f_n)=\begin{cases}
\dfrac{n^2}{2}\displaystyle\prod_{p\mid n}\left(1-\dfrac{1}{p^2}\right) = \dfrac{|V_n|}{2} & n>2,\\[1.2em]
n^2\displaystyle\prod_{p\mid n}\left(1-\dfrac{1}{p^2}\right)=|V_n| & n = 2.
\end{cases}
\]

By the orbit--degree dictionary  (Proposition 34 in \cite{Bourdon_2024}), $x$ corresponds to the orbit $H(n)\cdot v$ for some vector $v\in V_n$, and one has 
\[
\deg(x)=\frac{\#(H(n)\cdot v)}{2}\quad (n>2)~,\qquad \deg(x)=\#(H(n)\cdot v)\quad (n = 2)~.
\]
Hence, we see that $\deg(x)=\deg(f_n)$ if and only if $\#(H(n)\cdot v)=|V_n|$, as desired.
\end{proof}

In particular, for $p$ prime, a closed point $x\in X_1(p)$ with $j(x)=j(E)$ is not primitive if and only if $H(p)$ acts transitively on $V_p$. Serre's Uniformity Conjecture implies that there is no primitive point at level $p$ for $p>37$ prime.

Using this lemma, we can establish the following criterion for there being a unique primitive point attached to $E$.

\begin{theorem}\label{uniqueness criterion}
Let $E$ be a non-CM elliptic curve defined over $\mathbb Q$. The following are equivalent:
\begin{enumerate}
    \item $E$ has a unique primitive point attached to it.
    \item $H(n)$ acts transitively on $V_n$ for all $n\mid m_0$.
    \item $H(m_0)$ acts transitively on $V_{m_0}$.
    \item There is a unique closed point $x\in X_1(m_0)$ with $j(x)=j(E)$.
\end{enumerate}
\end{theorem}

\begin{proof}
Note the unique closed point $x'\in X_1(1)$ above $j(E)$ is always primitive as no natural projection map onto a lower level exists. Hence, $E$ has a unique primitive point attached to it if and only if $x'$ is the (unique) primitive point associated to every other closed point above $j(E)$, which is if and only if the natural map $f_n:X_1(n)\rightarrow X_1(1)$ satisfies $\deg(x)=\deg(f_n)\cdot \deg(x')$ for all $x\in X_1(n)$ with $j(x)=j(E)$ and for all $n\mid m_0$ with $n>1$. Now applying the above lemma (noting $\deg(x')=1$) gives $(1)\Leftrightarrow (2) \Rightarrow (3) \Leftrightarrow (4)$.

The final step is to show $(3)\Rightarrow (2)$. Arguing by contrapositive, suppose $H(n)$ does not act transitively on $V_n$ for some $n\mid m_0$. By Lemma \ref{some equivalent statements}, this means there exists an $n\mid m_0$ such that for all closed points $x\in X_1(n)$ with $j(x)=j(E)$ we have $\deg(x)<\deg(f_n)$. Consider the natural map $f:X_1(m_0)\rightarrow X_1(n)$, and say $f(x'')=x$.  Then, combining the facts that $\deg(x)<\deg(f_n)$ and $\deg(x'')\le \deg(f)\cdot \deg(x)$,  we get \[\deg(x'')<\deg(f)\cdot \deg(f_n) = \deg(f_{m_0})~,\]
which implies $H(m_0)$ does not act transitively on $V_{m_0}$ by Lemma \ref{some equivalent statements}, as we needed to show.
\end{proof}
\subsection{Sufficient Criteria}
Next, we isolate two concrete hypotheses, local transitivity at each prime power dividing $m_0$ and a stabilizer-surjectivity strengthening of entanglement-freeness for coprime division fields. The hypotheses imply transitivity of $G(n)\subseteq H(n)$ on $V_n$ for every $n \mid m_0$, hence a unique primitive point by Theorem \ref{uniqueness criterion}.

Fix a non-CM elliptic curve $E/\QQ$. For each $n\geq 1$, write $K_n:=\QQ(E[n[).$ We define the two following statements: \\
$(\LT_G) \quad$ For every prime power $\ell^k \mid m_0$, $G(\ell^k)$ acts transitively on $V_{\ell^k} \subseteq (\ZZ/\ell^k\ZZ)^2.$
$(\EF_{m_0})\quad$ For every coprime $a,b$ with $ab \mid m_0,$ let $L_{a,b} \coloneq K_a \cap K_b$ and $Q_{a,b} \coloneq \Gal(L_{a,b}/\QQ)$. Under restriction, the natural surjections $$\Gal(K_a/\QQ) \twoheadrightarrow Q_{a,b} \quad \Gal(K_b/\QQ) \twoheadrightarrow Q_{a,b}$$ are surjective stabilizers of exact-order points: for any $P_a \in E[a]$ and $P_b \in E[b]$ of exact order $a,b$ (respectively), $$\text{res}_a (\text{Stab}_{\Gal(K_a/\QQ)}(P_a))=Q_{a,b} \qquad \text{res}_{b}(\text{Stab}_{\Gal(K_b/\QQ)}(P_b))=Q_{a,b}.$$ 

The resulting uniqueness statement is recorded next.

\begin{theorem}\label{sufficient criteria}
    Assume $(\LT_G)$ and $(\EF_{m_0})$ hold. Then, the set $\PE$ of primitive points associated to a non-CM elliptic curve $E/\QQ$ has only one primitive point.
\end{theorem}

\begin{proof}
    Let $a,b \geq 1$ be coprime. Chinese Remainder Theorem gives $E[ab] \cong E[a] \times E[b]$ over $\QQ$, hence $K_{ab}=K_aK_b$. With $L_{a,b}=K_a \cap K_b$ and $Q_{a,b}=\Gal(L_{a,b}/\QQ),$ restriction yields $$\Gal(K_{ab}/\QQ) \cong \Gal(K_a/\QQ) \times_{Q_{a,b}} \Gal(K_b/\QQ).$$ Identifying $\Gal(K_t/\QQ) \cong G(t)$ for $t \in \{a,b,ab\}$, one gets: Under the CRT identification $$\GL_2(\ZZ/ab\ZZ) \cong \GL_2(\ZZ/a\ZZ) \times \GL_2(\ZZ/b\ZZ),$$ the image of $G(ab)$ equals the fiber product $G(a) \times_{Q_{a,b}}G(b)$. Let $A \curvearrowright X$ and $B \curvearrowright Y$ be groups acting transitively on sets, and let $\phi: A \twoheadrightarrow Q$, $\psi: B \twoheadrightarrow Q$ be surjections. Define $$A \times_{Q} B  \coloneq  \{(a,b) \in A \times B : \phi(a)=\psi(b)\},$$ acting on $X \times Y$ componentwise. Fix $x_0 \in X, y_0 \in Y$. If $\phi(\text{Stab}_A(x_0))=Q$, then $A \times_{Q} B$ is transitive on $X \times Y$. Indeed, given $(x,y) \in X \times Y$, choose $a \in A$ with $ax_0=x$ and $b \in B$ with $by_0=y.$ Set $q \coloneq \phi(a)^{-1}\psi(b) \in Q$. Choose $s \in \text{Stab}_{A}(x_0)$ with $\phi(s)=q$. Then, $$\phi(as)=\phi(a)\phi(s)=\phi(a)\phi(a)^{-1}\psi(b)=\psi(b),$$ so $(as,b) \in A \times_{Q}B$ and $(as,b)(x_0,y_0)=(x,y).$ Letting $A=G(a)$, $B=G(b)$, $X=V_a$, $Y=V_b$, and $Q=Q_{a,b}$. We then have that if $G(a)$ and $G(b)$ act transitively on $V_a$ and $V_b$ respectively, then $G(a) \times_{Q_{a,b}} G(b)$ acts transitively on $V_{ab}$. We now show that for every $n \mid m_0,$ the group $G(n)$ acts transitively on $V_n$. We prove by induction on $\omega(n),$ the number of distinct prime divisors of $n$. If $\omega(n)=1,$ then $n=\ell^k$, so the base case holds according to $(\LT_G).$ If $n=ab$ with $(a,b)=1$ and $\omega(a),\omega(b)<\omega(n),$ then $$V_{ab} \cong V_a \times V_b.$$ By induction, $G(a)$ and $G(b)$ are transitive on $V_a$ and $V_b$. The map $G(a) \twoheadrightarrow G_{a,b}$ and $G(b) \twoheadrightarrow Q_{a,b}$ are induced by restriction to $L_{a,b}$. $(\EF_{m_0})$ says these maps are surjective on stabilizers of exact-order points, hence on stabilizers of vectors in $V_a$ and $V_b$. Then, we get transitivity of $G(ab)$ on $V_{ab}$. Thus, $G(n)\subseteq H(n)$ is transitive on $V_n$ for all $n \mid m_0$. By Theorem \ref{uniqueness criterion}, this implies that there exists only one primitive point associated to $E$.
\end{proof}
We also record a more restrictive (and simpler) entanglement-freeness hypothesis:

$(\EF_{m_0}^*) \qquad$ For all coprime integers $a,b$ with $ab\mid m_0$, $K_a\cap K_b=\mathbb{Q}.$

Condition $(\EF_{m_0}^*)$ implies $(\EF_{m_0})$. Indeed, if $K_a\cap K_b=\mathbb{Q}$ then the associated entanglement quotient $Q_{a,b}$ is trivial, and consequently the stabilizer-surjectivity requirement appearing in $(\EF_{m_0})$ is automatic. In particular, either of the hypotheses $(\EF_{m_0})$ or $(\EF_{m_0}^*)$, when combined with $(\LT_G)$, yields a sufficient criterion for $E$ to have a unique primitive point. Since $(\EF_{m_0}^*)$ is strictly stronger than $(\EF_{m_0})$, one expects $(\EF_{m_0})$ to hold for a larger class of elliptic curves.

\begin{corollary}
    Assume $(\LT_G)$ and $(\EF_{m_0}^*)$ hold. Then, the set $\PE$ of primitive points associated to a non-CM elliptic curve $E/\QQ$ has only one primitive point.
\end{corollary}
We provide some motivation behind Theorem \ref{sufficient criteria}: The primitive point problem, as we have seen thus far, is governed by a single mechanism, which is the orbit structure. In general, \ref{uniqueness criterion} has shown that transitivity on all levels dividing $m_0$ prevents any closed point above $j(E)$ above level 1 from being a primitive point. The local transitivity hypothesis $\LT_G$ addresses this orbit problem at prime powers: it asserts that $G(\ell^k)$ is already transitive on $V_{\ell^k}$. The obstruction to extending this to general $n$ is the failure of the mod-$ab$ image to decompose as a direct product when $(a,b)=1$. In general, one has only a fiber-product description $$G(ab)\subset G(a) \times G(b)$$ controlled by the entanglement field $K_a \cap K_b$. Thus the global difficulty is, even when $G(a)$ and $G(b)$ are individually transitive on $V_a$ and $V_b$ respectively, the compatibility constraint defining the fiber product can, in principle, split $V_{ab} \cong V_a \times V_b$ into multiple orbits. this mechanism leads to entanglement creating multiple closed points above $j$ on $X_1(ab)$.

$\EF_{m_0}$ is defined to isolate minimal additional input needed to prevent that orbit-splitting. Instead of being a ``no entanglement'' condition, it is a stabilizer-level condition that guarantees the fiber product still acts transitively on primitive level structure, hence that entanglement still present in $K_a \cap K_b$ is irrelevant for the primitive-point problem.

What $\EF_{m_0}$ is, is that, if we fix coprimes $a,b \mid m_0$ and let $L_{a,b}$ and $Q_{a,b}$ be as defined previously, under the CRT identification, the mod-$ab$ image is the fiber product $$G(ab) \cong G(a) \times_{Q_{a,b}} G(b)$$ where the maps to $Q_{a,b}$ are given by restriction to $L_{a,b}$. The fiber-product constraint is a matching condition in the quotient $Q_{a,b}$. $\EF_{m_0}$ asserts that this matching condition can always be resolved within a point stabilizer: for any point $P_a$ of exact order $a,$ the restriction map to $Q_{a,b}$ is surjective on $\text{Stab}(P_a)$, and similarly on the $b$-side. Equivalently, by Galois correspondence, if $k(P_a)$ denotes the field of definition of $P_a$ inside $K_a$, then $$k(P_a) \cap L_{a,b}=\QQ$$ and likewise for $P_b$. Thus, $\EF_{m_0}$ allows nontrivial entanglement at the level of full division fields while requiring that entanglement not be detected by the fields generated by individual points of exact order $a,b$. This is the geometric condition needed to keep the fiber of $X_1(ab) \to X_1(1)$ above $j$ a single orbit.

The standard entanglement viewpoint treats $L_{a,b}=K_a \cap K_b$ as the primary object: one studies its size or type and the induced subgroup $G(ab) \subseteq G(a) \times G(b)$. For the primitive-point problem, that perspective is too coarse: the existence of entanglement does not by itself determine whether $V_a$ or $V_b$ split into multiple orbits. $\EF_{m_0}$ refines the entanglement analysis to the precise granularity relevant to $X_1(n)$: it distinguishes ``harmful'' entanglement (which is visible in stabilizers and forces orbit-splitting on $V_n$) from ``harmless'' entanglement (which may exist in $K_a \cap K_b$ but does not affect orbits). This yields the remark that uniqueness of primitive point is controlled by stabilizer-detectable entanglement, not by entanglement per se. \\

We now state an algorithm that checks $(\LT_G)$.
\begin{algorithm}[H]
\caption{Transitivity on exact-order $\ell^k$ vectors $(\LT_G)$}
\label{alg:transitivity-lk}
\begin{algorithmic}[1]
\Statex \textbf{Input:} prime $\ell$, integer $k\ge 1$, generators $S\subseteq \GL_2(\ZZ/\ell^k\ZZ)$ for $G=\langle S\rangle$.
\Statex \textbf{Output:} \textbf{TRUE}/\textbf{FALSE} (and orbit/expected sizes).
\State $n\gets \ell^k$, $v_0\gets (1,0)\in (\ZZ/n\ZZ)^2$.
\State $O\gets$ the smallest subset of $(\ZZ/n\ZZ)^2$ containing $v_0$ and closed under $v\mapsto g v$ for all $g\in S$ \emph{(compute by closure/BFS)}.
\State $\textsf{orb}\gets |O|$, \quad $\textsf{exp}\gets n^2-(n/\ell)^2$ \Comment{$=\#\{v:\ord(v)=n\}$}
\State \Return $(\textsf{orb}=\textsf{exp},\,\textsf{orb},\,\textsf{exp})$.
\end{algorithmic}
\end{algorithm}
\begin{remark}
In the hypothesis $\mathrm{LT}_G$, it suffices to check transitivity only at the top prime powers $\ell^{e}\parallel m_0$. Indeed, if $e=v_\ell(m_0)$ and $G(\ell^{e})$ acts transitively on vectors of exact order $\ell^{e}$ in $(\mathbb Z/\ell^{e}\mathbb Z)^2$, then for every $1\le k\le e$ the reduction map
\[
\pi_{e,k}:\GL_2(\mathbb Z/\ell^{e}\mathbb Z)\to \GL_2(\mathbb Z/\ell^{k}\mathbb Z)
\]
sends $G(\ell^{e})$ onto $G(\ell^{k})$. Given $v,w\in (\mathbb Z/\ell^{k}\mathbb Z)^2$ of exact order $\ell^{k}$, choose lifts $\tilde v,\tilde w\in (\mathbb Z/\ell^{e}\mathbb Z)^2$ with $\tilde v\equiv v$ and $\tilde w\equiv w\pmod{\ell^{k}}$. Since $v,w$ are not divisible by $\ell$, the lifts $\tilde v,\tilde w$ have exact order $\ell^{e}$. By transitivity at level $\ell^{e}$ there exists $g\in G(\ell^{e})$ with $g\tilde v=\tilde w$, and reducing modulo $\ell^{k}$ gives $\pi_{e,k}(g)\in G(\ell^{k})$ with $\pi_{e,k}(g)v=w$. Thus $G(\ell^{k})$ is transitive on exact-order $\ell^{k}$ vectors. Consequently, requiring transitivity for all $\ell^{k}\mid m_0$ is equivalent to requiring it only for $\ell^{v_\ell(m_0)}\parallel m_0$.
\end{remark}

We will later state and prove the algorithm to check $\EF_{m_0}$. First, we prove an alternate formulation of the condition:
\begin{proposition}\label{equivalence conditions}
    Let $E/\QQ$ be a non-CM elliptic curve and let $m_0 \geq 1$. Fix coprime integers $a,b >1$ with $ab \mid m_0$, and let $K_n$, $L_{a,b}$ be defined as above. Let $A \coloneq G(a)$, $B \coloneq G(b)$, and let $\Gamma\subseteq A\times B$ be the image of $G(ab)$ under the CRT identification \[ \GL_2(\ZZ/ab\ZZ)\;\cong\;\GL_2(\ZZ/a\ZZ)\times \GL_2(\ZZ/b\ZZ). \]
    Define \[  N_A \coloneq \{x\in A:(x,I)\in \Gamma\},\qquad N_B \coloneq \{y\in B:(I,y)\in \Gamma\}.\] Then the following hold.
    \begin{enumerate}
    \item There are canonical isomorphisms \[
    A/N_A \;\cong\; Q_{a,b} \;\cong\; B/N_B. \]
    \item For any point $P_a\in E[a]$ of exact order $a$ and any point $P_b\in E[b]$ of exact order $b$, the field-theoretic stabilizer-surjectivity conditions
    \[
    \text{res}_a\bigl(\text{Stab}_{\Gal(K_a/\QQ)}(P_a)\bigr)=Q_{a,b},\qquad
    \text{res}_b\bigl(\text{Stab}_{\Gal(K_b/\QQ)}(P_b)\bigr)=Q_{a,b}
    \]
    are equivalent to the group-theoretic conditions
    \[
    \langle \text{Stab}_A(v),\,N_A\rangle=A,\qquad
    \langle \text{Stab}_B(w),\,N_B\rangle=B,
    \]
    where $v\in V_a$ and $w\in V_b$ correspond to $P_a,P_b$ under a choice of $\ZZ/a\ZZ$-basis of $E[a]$ and $\ZZ/b\ZZ$-basis of $E[b]$.
    \end{enumerate}
\end{proposition}
\begin{proof}
    Throughout, we use the standard Galois correspondence for the finite Galois extensions $K_a/\QQ$, $K_b/\QQ,$ and $K_{ab}/\QQ$. The action of $\GalQ$ on $E[n]$ gives a surjection $\rho_{E,n}: \GalQ \twoheadrightarrow G(n) \subseteq \GL_2(\Zn),$ and by construction its kernel is $\Gal(\overline{\QQ}/K_n)$. Hence, $G(n) \cong \Gal(K_n/\QQ).$ In particular, we fix the identifications $A \coloneqq \Gal(K_a/\QQ)$ and $B \coloneq \Gal(K_b/\QQ).$ For $ab$, we likewise have $G(ab) \cong \Gal(K_{ab}/\QQ),$ and under Chinese Remainder Theorem, the restriction map $$\Gal(K_{ab}/\QQ)\rightarrow \Gal(K_a/\QQ)\times \Gal(K_b/\QQ)$$ corresponds (under the above identifications) to the map $G(ab) \to A \times B$ sending $g \mapsto (g \bmod a, g \bmod b)$. Thus, $\Gamma \subseteq A \times B$ is precisely the image of $\Gal(K_{ab}/\QQ)$ in $\Gal(K_a/\QQ)\times\Gal(K_b/\QQ).$

    Let $\varphi_a: \Gal(K_{ab}/\QQ) \to \Gal(K_a/\QQ)$ and $\varphi_b: \Gal(K_{ab}/\QQ) \to \Gal(K_{b}/\QQ)$ be restriction. By definition, $$N_A=\{\,\varphi_a(\sigma): \sigma\in \Gal(K_{ab}/\QQ),\ \varphi_b(\sigma)=1\,\}.$$ The condition $\varphi_b(\sigma)=1$ is equivalent to $\sigma$ acting trivially on $K_b$; hence such $\sigma$ lies in $\Gal(K_{ab}/K_b).$ Therefore $N_A=\varphi_{a}\bigl(\Gal(K_{ab}/K_b)\bigr.$ Now note that $\varphi_a$ maps $\Gal(K_{ab}/K_b)$ isomorphically onto $\Gal(K_a/L_{a,b})$: indeed, elements of $\Gal(K_{ab}/K_b)$ act on $K_a$ and fix $K_b$, so they fix $L_{a,b}=K_a\cap K_b$ and hence restrict to $\Gal(K_a/L_{a,b})$; conversely any automorphism of $K_a$ fixing $L_{a,b}$ extends (by acting trivially on $K_b$) to an automorphism of the compositum $K_{ab}$ fixing $K_b$. Thus, under the identification $A\cong \Gal(K_a/\QQ)$, we have \[N_A\;\cong\;\Gal(K_a/L_{a,b}).\] Similarly, \[N_B\;\cong\;\Gal(K_b/L_{a,b}).\] The restriction map $\varphi_a: \Gal(K_a/\QQ) \to \Gal(L_{a,b}/\QQ)$ is surjective because $L_{a,b} \subseteq K_a$ is Galois over $\QQ$ as it is the intersection of two Galois extensions, and its kernel is $\Gal(K_a/L_{a,b})$. Transporting this statement through $A$ and $N_A \cong \Gal(K_a/L_{a,b})$ yields a canonical isomorphism $$A/N_A \cong \Gal(L_{a,b}/\QQ)=Q_{a,b}.$$ The same argument gives $B/N_B \cong Q_{a,b}.$ This proves  $(1).$

    Fix $P_a \in E[a]$ of exact order $a$. Choose a $\ZZ/a\ZZ$-basis of $E[a]$ so that $P_a$ corresponds to $v \in V_a$. Under $A \coloneq \Gal(K_a/\QQ),$ the stabilizer subgroup \[ \text{Stab}_{\Gal(K_a/\QQ)}(P_a)=\{\sigma\in \Gal(K_a/\QQ):\sigma(P_a)=P_a\}\] corresponds exactly to \[ \text{Stab}_A(v)=\{g\in A: gv=v\}.\] Consider the surjection $A\twoheadrightarrow A/N_A\cong Q_{a,b}$ from $(1)$. Its restriction to $\text{Stab}_A(v)$ is surjective if and only if the image of $\text{Stab}_A(v)$ in $A/N_A$ equals all of $A/N_A$. Equivalently, \[\text{Stab}_A(v)\cdot N_A = A\qquad\Longleftrightarrow\qquad\langle \text{Stab}_A(v),N_A\rangle=A,\] since $N_A\triangleleft A$. Transporting back through the identifications $A/N_A\cong Q_{a,b}$ and $A\cong \Gal(K_a/\QQ)$, this is exactly \[\text{res}_a\bigl(\text{Stab}_{\Gal(K_a/\QQ)}(P_a)\bigr)=Q_{a,b}.\] The same argument applies to $b$ and proves $(2)$.
\end{proof}
For each coprime $a,b$ with $ab\mid m_0$, the field-theoretic stabilizer-surjectivity condition in $(\EF_{m_0})$ is equivalent to the group-theoretic stabilizer-generation condition expressed using $A,B,\Gamma,N_A,N_B$. It thus suffices to check, for each coprime pair $(a,b)$ where $ab \mid m_0$, whether the stabilizer-surjectivity conditions $$\langle \text{Stab}_{G(a)}(v), N_a\rangle=G(a) \quad \langle \text{Stab}_{G(b)}(w),N_b \rangle = G(b)$$ for all $v \in V_a$ and $w \in V_b$ hold. 

\medskip

\begin{algorithm}[H]
\caption{Stabilizer-Surjectivity $(\EF_{m_0})$}
\label{stabilizer-surjectivity}
\begin{algorithmic}[1]
\Statex \textbf{Input:} Generators of $G(m_0)\le \GL_2(\ZZ/m_0\ZZ)$.
\Statex \textbf{Output:} \textbf{PASS}/\textbf{FAIL}.
\State For each coprime $a,b>1$ with $ab\mid m_0$, form $(A,B,\Gamma,N_A,N_B)$.
\State \textbf{FAIL} if any pair violates (i) $[A:N_A]=[B:N_B]$, (ii) $\langle \text{Stab}_A(v),N_A\rangle=A$ for all $[v]\in V_a/A$, or (iii) $\langle \text{Stab}_B(w),N_B\rangle=B$ for all $[w]\in V_b/B$.
\State Otherwise \textbf{PASS}.
\end{algorithmic}
\end{algorithm}

\begin{lemma}\label{lem:orbit-invariance}
Let $a,b>1$ be coprime with $ab\mid m_0$. Let $\Gamma\subseteq G(a)\times G(b)$ be the CRT-image of $G(ab)$, with projections
$\pi_a,\pi_b$. Define
\[
N_A \coloneq \{x\in G(a):(x,1)\in\Gamma\}=\pi_a(\ker\pi_b)\] \[
N_B \coloneq \{y\in G(b):(1,y)\in\Gamma\}=\pi_b(\ker\pi_a).
\]
Then $N_A\triangleleft G(a)$ and $N_B\triangleleft G(b)$. Moreover, for $t\in\{a,b\}$ the condition
\[
\langle \text{Stab}_{G(t)}(v),\,N_t\rangle = G(t)
\]
is constant on $G(t)$-orbits in $V_t$; hence it suffices to test one representative per orbit.
\end{lemma}

\begin{proof}
Since $\pi_a,\pi_b$ are surjective homomorphisms, $\ker(\pi_b)\triangleleft \Gamma$ and thus
$N_A=\pi_a(\ker\pi_b)\triangleleft \pi_a(\Gamma)=G(a)$; similarly $N_B\triangleleft G(b)$.
If $v'=gv$ with $g\in G(t)$, then $\text{Stab}(v')=g\,\text{Stab}(v)\,g^{-1}$, and normality of $N_t$ gives
$$\langle \text{Stab}(v'),N_t\rangle = g\langle \text{Stab}(v),N_t\rangle g^{-1}.$$
\end{proof}

By Proposition~\ref{equivalence conditions} (2), the algorithm checks $\EF_{m_0}$ by verifying
$\langle\text{Stab}_A(v),N_A\rangle=A$ (resp. $\langle\text{Stab}_B(w),N_B\rangle=B$) on orbit representatives,
which is sufficient by Lemma~\ref{lem:orbit-invariance}.

\begin{ex}
    We take $E/\QQ$ to be the curve \href{https://www.lmfdb.org/EllipticCurve/Q/232544/f/1}{232544.f1}. Applying Algorithms \ref{alg:transitivity-lk} and $\ref{stabilizer-surjectivity}$ shows that $E$ meets both $(\LT_G)$ and $(\EF_{m_0}).$ Thus, $\lvert \PE \rvert=1.$
\end{ex}
\begin{ex}\label{counterex}
    Satisfying $(\LT_G)$ and $(\EF_{m_0})$ is not a necessary condition for uniqueness of primitive point. Let $E/\QQ$ be the non-CM elliptic curve \href{https://www.lmfdb.org/EllipticCurve/Q/1944/c/1}{1944.c1}. This curve has $N=m_0=12,$ and $\lvert \PE \rvert=1$. Applying Algorithm \ref{stabilizer-surjectivity} to $G(m_0)$ shows that the curve fails $(\EF_{m_0})$, at both pairs $(a,b) \in \{(2,3),(3,4)\}.$ We also note that for these two pairs, $K_a \cap K_B \neq \QQ$.
\end{ex}
We will later prove that if $E$ is Serre, then $\lvert \PE\rvert=1$; therefore, we will only consider $(\LT_G)$ and $(\EF_{m_0})$ in the context of non-Serre curves, which are curves with $I(E)>2$.

\begin{remark}
    Jones shows that Serre curves have density $1$ in the naive height ordering; in particular the non-Serre locus is thin. The purpose of the criteria in this section is not to optimize density statements, but to give an effective and verifiable mechanism that detects (and in favorable cases forces) uniqueness of primitive points for specific non-Serre curves, where the phenomenon is arithmetically nontrivial and computationally relevant.
\end{remark}
In practice, it is not necessary to run both algorithms. In fact, to speed up the computation, we will only run Algorithm \ref{stabilizer-surjectivity} if Algorithm \ref{alg:transitivity-lk} returns \textbf{PASS}. We give an overview of the main algorithm:
\begin{algorithm}[H]
\caption{Main algorithm}
\label{alg:main-alg}
\begin{algorithmic}[1]
\Statex \textbf{Input:} a non-CM $j$-invariant $j\in\QQ$
\Statex \textbf{Output:} \textbf{PASS} or \textbf{FAIL}
\State Choose an elliptic curve $E/\QQ$ with $j(E)=j$.
\State Compute the adelic image $G=\rho_E(\Gal(\overline{\QQ}/\QQ))$ via Zywina's algorithm \cite{zywina2024explicitopenimageselliptic}.
\State Compute $m_0$ using \cite[Algorithm~2]{Bourdon_2024}, and compute $G(m_0)$ and $G(\ell^k)$ for all $\ell^k\parallel m_0$.
\If{Algorithm~\ref{alg:transitivity-lk} applied to $G(\ell^k)$ returns \textbf{FAIL} for some $\ell^k\parallel m_0$}
    \State \Return \textbf{FAIL}
\EndIf
\If{Algorithm~\ref{stabilizer-surjectivity} applied to $G(m_0)$ returns \textbf{FAIL}}
    \State \Return \textbf{FAIL}
\EndIf
\State \Return \textbf{PASS}
\end{algorithmic}
\end{algorithm}
\thmsubsection{Complexity of Algorithm 1}
Algorithm \ref{alg:transitivity-lk} performs BFS inside $(\ZZ/\ell^k\ZZ)^2.$ The explored orbit has size at most $\lvert V_{\ell^k} \rvert = \ell^{2k}-\ell^{2k-2}.$ If there are $g$ generators of $G(\ell^k)$, then the runtime is $O(g \lvert V_{\ell^k}\rvert)=O(g\ell^{2k})$ with memory $O(\lvert V_{\ell^k}\rvert).$

\thmsubsection{Complexity of Algorithm 2}
Algorithm \ref{stabilizer-surjectivity} loops over coprime pairs $a,b>1$ with $ab \mid m_0$, forms $\Gamma \subseteq G(a) \times G(b)$ and $N_A, N_B$, and then checks $\langle \text{Stab}_A(v),N_A)=A$ on orbit representatives (and similarly for $B$). Lemma \ref{lem:orbit-invariance} reduces the number of stabilizer checks to the number of $G(a)$-orbits in $V_a$ and $G(b)$-orbits in $V_b$. In the worst case this is $O(\lvert V_a\rvert +\lvert V_b\rvert),$ so an overall coarse bound is $$O\left(\sum_{ab\mid m_0, (a,b)=1} \left( \lvert V_a\rvert +\lvert V_b\rvert \right) \cdot C_{\text{stab}}\right)$$ where $C_{stab}$ is the cost of computing stabilizers and subgroup generation inside $G(a)$ and $G(b)$. The work scales essentially like a weighted sum of $a^2$ over coprime divisor pairs, not like any computation on modular curves themselves.

\thmsubsection{Results}
We stratify curves by adelic index using the conjectured finite list in
\cite{zywina2024explicitopenimageselliptic} and ran Algorithm \ref{alg:main-alg} on all non-CM,
non-Serre elliptic curves in \cite{lmfdb} whose adelic indices lie in this list,
with the exception of index \(12\), for which we processed the first quarter of the curves, totalling up to over $800\,000$ curves. For each curve we record its \texttt{LMFDB label}, \(m_0\), \(I(E)\), \(|\PE|\),
the algorithm \texttt{output}, the \texttt{fail stage} (LT$_G$ or EF$_{m_0}$), the
\texttt{first failure witness} (for LT$_G$ and, conditional on passing LT$_G$, for
EF$_{m_0}$), and the \texttt{runtime}.

In this corpus, \(|\PE|=1\) occurs for \(1.62\%\) of curves. The certificate returns
\texttt{PASS} for \(1.49\%\) of curves and exhibits no false positives: among curves with
\(|\PE|\neq 1\), \texttt{PASS} never occurs. On the true class, recall is \(92.45\%\), i.e.
the false negative (FN) rate is \(7.55\%\).

Fix a prime power \(\ell^k\parallel m_0\). A first-fail witness \((\ell,k)\) indicates that
the mod-\(\ell^k\) image does not act transitively on \(V_{\ell^k}\); equivalently, the
Galois action on cyclic subgroups of order \(\ell^k\) splits into multiple orbits. This is a
local ``inner-structure'' defect: there is residual structure at level \(\ell^k\) obstructing
full mixing of primitive \(\ell^k\)-torsion directions. Empirically, this phenomenon is
overwhelmingly \(2\)-adic: \((2,1),(2,2),(2,3)\) account for \(88.54\%\) of all LT$_G$
failures, while \((3,1)\), \((3,2)\), and \((5,1)\) contribute on a much smaller scale. Thus, on the full corpus, negative instances are explained predominantly by
low-level prime-power non-mixing, dominated by the first few \(2\)-power layers.

EF$_{m_0}$ is evaluated only after every relevant prime power passes LT$_G$, so it probes a
different layer: incompatibility across coprime moduli \((a,b)\), detected through the CRT
image \(\Gamma\subseteq G(a)\times G(b)\) together with a stabilizer-generation test on one
side. In field-theoretic terms, nontrivial intersections \(K_a\cap K_b\) can force the image
\(G(ab)\) to be a proper fiber product inside \(G(a)\times G(b)\), constraining orbit growth
even when each prime-power component is locally mixing. This coprime-coupling phenomenon is
extremely sparse: only \(0.14\%\) of the corpus fails EF$_{m_0}$.

Nevertheless, EF$_{m_0}$ failures are rigidly supported on a small set of witness types, and
these witnesses split into distinct empirical phenotypes. Writing \(\mathrm{FN\ share}\) for
the proportion of
$\EF_{m_0}$ failures in a witness class that actually have \(|\PE|=1\), we obtain:
\[
\begin{array}{c|ccccc|c|c}
(a,b) & (2,3) & (2,9) & (3,8) & (2,3) & (7,8) & (3,4) & (2,5) \\
\mathrm{side} & A & A & A & B & A & A & A \\
\hline
\#\mathrm{EF\_FAIL} & 418 & 246 & 188 & 56 & 16 & 132 & 116 \\
\#\mathrm{FN}       & 418 & 246 & 188 & 56 & 16 & 108 & 0 \\
\mathrm{FN\ share}  & 1.00 & 1.00 & 1.00 & 1.00 & 1.00 & 0.82 & 0.00
\end{array}
\]
The dominant \(2\)--\(3\)-power witnesses \((2,3)\), \((2,9)\), \((3,8)\) on side $A$ (and \((2,3)\) on
side \(B\)) are purely conservative on this corpus: whenever EF$_{m_0}$ fails for these
types, one still has \(|\PE|=1\). In contrast, the witness \((2,5)\) is decisive on this
corpus: whenever EF$_{m_0}$ detects the \((2,5)\) coprime obstruction, one always has
\(|\PE|>1\). The witness \((3,4)\) is mixed (FN share \(=0.82\)), accounting for both false
negatives and genuine non-uniqueness. EF false negatives are strongly one-sided: \(94.57\%\)
occur on side \(A\). Moreover, $\EF_{m_0}$ witnesses occur only at small moduli forced by
\(m_0\)-divisibility (e.g. \(((2,3),A)\) only for \(m_0\in\{6,12,30\}\), \(((2,9),A)\) only
for \(m_0\in\{18,36,72\}\), \(((3,8),A)\) only for \(m_0\in\{24,72\}\), \(((3,4),A)\) only
for \(m_0\in\{12,24\}\), and \(((2,5),A)\) only for \(m_0\in\{10,20\}\)); thus the
empirically relevant coprime-coupling layer is concentrated at low levels.

False negatives occur only at
\(I(E)\in\{6,12,24,30,36,48,84,120,144\}\), and several strata exhibit rigid $\EF_{m_0}$ behavior on
the true class. For instance, at \(I(E)=6\) the true cases all fail $\EF_{m_0}$ with witness
\(((2,3),A)\) (with \(m_0=6\)), while at \(I(E) \in \{4,8,10,20,40,54\}\) the certificate has 100\%
recall on the true class. This is consistent with the witness dichotomy above: once local
prime-power mixing holds, the remaining barrier to certification is governed by a small menu
of low-level coprime constraints, and some of those constraints persist uniformly across an
index stratum even though they do not correspond to genuine non-uniqueness.

Runtime quantiles were recorded for the curves that reached the $\EF_{m_0}$ stage (i.e. \texttt{PASS}
or EF\_FAIL). On this subset, the median runtime is 0.293s (90th percentile 13.870s, 95th
percentile 57.819s). Certified cases have median 0.281s (95th percentile 30.984s), while
EF-failing cases show a heavy tail (median 0.586s, 90th percentile 226.651s, 95th percentile
508.415s).

\section{The Case of Serre Curves}
Recall the adelic Galois representation
\[
\rho_E:\Gal_{\mathbb{Q}}\rightarrow \GL_2(\widehat{\mathbb{Z}})
\]
and write $I(E)=[\GL_2(\widehat{\mathbb{Z}}):\rho_E(\Gal_{\mathbb{Q}})]$ for the index of its image. An elliptic curve $E/\mathbb{Q}$ is called a \emph{Serre curve} if $I(E)=2$, i.e.\ if its adelic Galois representation has the largest possible image among non-CM elliptic curves. The significance of Serre curves is underscored by the result established in \cite{Jones2009-ob}, which states that, when elliptic curves are ordered by height, Serre curves have density one. This means that almost all elliptic curves over $\QQ$ are Serre curves. In this section we analyze the primitive points associated to Serre curves and prove that if $E$ is a Serre curve, then
\(
|\mathcal{P}(E)|=1.
\)
\smallskip
Let $G, G(n)$ be as defined in previous sections and let $H$ be an arbitrary group. For the rest of this section, unless stated otherwise, denote $H'\coloneq[H,H],$ the closed commutator subgroup of $H.$
\begin{theorem}
    Let $E/\mathbb{Q}$ be a non-CM elliptic curve. Let $m$ and $m_0$ be as previously defined. If $m_0=1$, then $E$ is a Serre curve.
\end{theorem}
\begin{remark}
    We note that this is the case where $E$ trivially has only 1 primitive point. If $m_0=1,$ then it has only 1 divisor, namely 1, and thus $X_1(1)$ is the only modular curve that can contain any primitive point associated with $E$, that is, the unique point above $j(E)$ on $X_1(1)$.
\end{remark}
\begin{proof}
Assume $m_0=1$. This implies that the $m$-adic image is full:
\[
\rho_{E,m^\infty}(\Gal_\QQ)=\GL_2(\ZZ_m)=\prod_{\ell\mid m}\GL_2(\ZZ_\ell).
\] Let $G(m)$ denote the image of $G$ under reduction modulo $m$. By \cite{zywina2022possibleindice} one has
\[
[\GL_2(\Zhat):G]=[\SL_2(\Zhat):G']=[\SL_2(\ZZ_m):G(m)'].
\]
Thus it suffices to show that $[\SL_2(\ZZ_m):G(m)']=2$.

Since $\rho_{E,m^\infty}(\Gal_\QQ)=\GL_2(\ZZ_m)$, we have $G(m)=\GL_2(\ZZ_m)$ and hence
\[
G(m)'=\GL_2(\ZZ_m)'=\prod_{\ell\mid m}\GL_2(\ZZ_\ell)'.
\] For odd primes $\ell$, Lang and Trotter show that $\GL_2(\ZZ_\ell)'=\SL_2(\ZZ_\ell)$, while for $\ell=2$ one has
\[
\GL_2(\ZZ_2)'=\ker(\mathrm{sgn})\cap \SL_2(\ZZ_2),
\]
where $\mathrm{sgn}:\GL_2(\ZZ_2)\to \GL_2(\FF_2)\cong S_3\to\{\pm 1\}$ is the sign character
(\cite{LangTrotter1976Frobenius}, Part II, \S3, Lemma 1 and Part III, \S2, \S4). In particular,
\[
[\SL_2(\ZZ_2):\GL_2(\ZZ_2)']=2,\qquad
[\SL_2(\ZZ_\ell):\GL_2(\ZZ_\ell)']=1\ \ (\ell \ \text{odd}).
\]
Using $\SL_2(\ZZ_m)=\prod_{\ell\mid m}\SL_2(\ZZ_\ell)$, we obtain
\[
[\SL_2(\ZZ_m):G(m)']
=\prod_{\ell\mid m}[\SL_2(\ZZ_\ell):\GL_2(\ZZ_\ell)']
=2.
\]
Therefore $[\GL_2(\Zhat):G]=2$, i.e.\ $E$ is a Serre curve.
\end{proof}

\begin{remark}
     We note that the construction of $m$ in \cite{zywina2022possibleindice} differs from the construction of $m$ in this paper by a factor of 5 if the $5$-adic Galois representation is surjective. This, however, does not meaningfully change the equality chain $[\GL_2(\Zhat):G]=[\SL_2(\Zhat):G']=[\SL_2(\ZZ_m):G(m)'],$ as we have demonstrated that $[\SL_2(\ZZ_5):\GL_2(\ZZ_5)']=1$.
\end{remark}
\begin{theorem}
Let $E/\mathbb{Q}$ be a Serre curve. Then $\lvert \mathcal{P}(E)\rvert=1$.
\end{theorem}

\begin{proof} Since $E$ is a Serre curve, one has $[\GL_2(\widehat{\mathbb{Z}}):G]=2$. In particular, $G$ is a normal subgroup of $\GL_2(\widehat{\mathbb{Z}})$, and the quotient $\GL_2(\widehat{\mathbb{Z}})/G$ is abelian. Hence the commutator subgroup
\[
C\coloneq[\GL_2(\widehat{\mathbb{Z}}),\GL_2(\widehat{\mathbb{Z}})]
\]
is contained in $G$. Let $C(n)$ be the image of $C$ in $\GL_2(\mathbb{Z}/n\mathbb{Z})$. Since the reduction map $\GL_2(\widehat{\mathbb{Z}})\twoheadrightarrow \GL_2(\mathbb{Z}/n\mathbb{Z})$ is surjective, one has
\[
C(n)=[\GL_2(\mathbb{Z}/n\mathbb{Z}),\GL_2(\mathbb{Z}/n\mathbb{Z})].\]
Consequently,
\(
C(n)\subseteq G(n)\subseteq H(n)
\) for all $n \geq 1$. We claim that $C(n)$ acts transitively on $V_n$; it then follows that $H(n)$ acts transitively on $V_n$ as well. Write $n=\prod_i p_i^{k_i}$. By the Chinese Remainder Theorem,
\[
\GL_2(\mathbb{Z}/n\mathbb{Z})\cong \prod_i \GL_2(\mathbb{Z}/p_i^{k_i}\mathbb{Z}),
\qquad
V_n\cong \prod_i V_{p_i^{k_i}},
\]
with componentwise action. Moreover, for a direct product $H=\prod_i H_i$ one has $[H,H]=\prod_i[H_i,H_i]$. Thus it suffices to prove that
\(
[\GL_2(\mathbb{Z}/p^k\mathbb{Z}),\GL_2(\mathbb{Z}/p^k\mathbb{Z})]
\)
acts transitively on $V_{p^k}$ for each prime power $p^k$. \\

\noindent
\emph{Case 1: $p=2$.}
For $k=1$, one has $V_2=\mathbb{F}_2^2\setminus\{0\}=\{(1,0),(0,1),(1,1)\}$. Let
\[
A=\begin{pmatrix}0&1\\1&0\end{pmatrix},\qquad
B=\begin{pmatrix}1&0\\1&1\end{pmatrix}\in \GL_2(\mathbb{F}_2).
\]
A direct computation gives
\[
g\coloneq ABA^{-1}B^{-1}=\begin{pmatrix}0&1\\1&1\end{pmatrix}\in [\GL_2(\mathbb{F}_2),\GL_2(\mathbb{F}_2)],
\]
and $g$ permutes the elements of $V_2$ as a $3$-cycle. Hence $[\GL_2(\mathbb{F}_2),\GL_2(\mathbb{F}_2)]$ acts transitively on $V_2$. Now assume $k\ge 1$ and set $R_k=\mathbb{Z}/2^k\mathbb{Z}$ and $C_k=[\GL_2(R_k),\GL_2(R_k)]$. The reduction map $R_{k+1}\twoheadrightarrow R_k$ induces a surjection $C_{k+1}\twoheadrightarrow C_k$. Suppose inductively that $C_k$ acts transitively on $V_{2^k}$. Fix $v\in V_{2^{k+1}}$. Choose $h\in C_k$ sending the reduction of $v$ modulo $2^k$ to the reduction of $(1,1)$, and lift $h$ to $\widetilde h\in C_{k+1}$. Then
\(
\widetilde h\,v \equiv (1,1)\pmod{2^k},
\)
so $\widetilde h\,v=(1,1)+2^k(\alpha,\beta)$ for some $\alpha,\beta\in\{0,1\}$.

Let
\[
U(t)\coloneq\begin{pmatrix}1&t\\0&1\end{pmatrix},\qquad
L(t)\coloneq\begin{pmatrix}1&0\\t&1\end{pmatrix}.
\]
Define $D_1=\mathrm{diag}(1+2^k,1)$ and $D_2=\mathrm{diag}(1,1+2^k)$ in $\GL_2(R_{k+1})$. A straightforward commutator computation yields
\[
[D_1,U(1)]=U(2^k)\in C_{k+1},\qquad [D_2,L(1)]=L(2^k)\in C_{k+1}.
\]
Evaluating on $(1,1)+2^k(\alpha,\beta)$, the element $U(2^k)^\alpha L(2^k)^\beta\in C_{k+1}$ adjusts the two coordinates independently by $2^k$ modulo $2^{k+1}$, and hence sends $(1,1)+2^k(\alpha,\beta)$ to $(1,1)$. Therefore some element of $C_{k+1}$ sends $v$ to $(1,1)$, proving transitivity of $C_{k+1}$ on $V_{2^{k+1}}$. This completes the induction. \\

\noindent
\emph{Case 2: $p$ odd.}
Let $R=\mathbb{Z}/p^k\mathbb{Z}$. Since $2\in R^\times$, one has for all $t\in R$ the identities
\[
[\mathrm{diag}(2,1),U(t)]=U(t),\qquad [\mathrm{diag}(1,2),L(t)]=L(t),
\]
so $U(t),L(t)\in[\GL_2(R),\GL_2(R)]$. Over the local ring $R$, the elementary unipotents $U(t)$ and $L(t)$ generate $\mathrm{SL}_2(R)$; hence
\(
\mathrm{SL}_2(R)\subseteq [\GL_2(R),\GL_2(R)].
\)

By Remark \ref{transitivity of SL_2}, we know $\mathrm{SL}_2(R)$ acts transitively on $V_{p^k}$, hence so does the larger group $C(p^k)=[\GL_2(R),\GL_2(R)]$. This shows that $C(n)$ acts transitively on $V_n$ for every $n$, and therefore so does $H(n)$. By Theorem \ref{uniqueness criterion}, we conclude that $\lvert \mathcal{P}(E)\rvert=1$.
\end{proof}
Since almost all elliptic curves are Serre by \cite{Jones2009-ob}, we have:
\begin{corollary}
    Almost all non-CM elliptic curves over $\QQ$ have exactly 1 associated primitive point.
\end{corollary}
We call $j(E)\in X_1(1)\cong \mathbb{P}^1$ an \emph{isolated $j$-invariant} if there exists an isolated point $x\in X_1(n)$ for some $n$ such that $j(x)=j(E)$.

\begin{corollary}
Let $E/\mathbb{Q}$ be a Serre curve. Then $j(E)$ is not isolated.
\end{corollary}

\begin{proof}
By the proposition, the only primitive point associated to $E$ is the degree-one point $x_1\in X_1(1)$ with $j(x_1)=j(E)$. Since $X_1(1)\cong \mathbb{P}^1$ has trivial Jacobian, the Abel--Jacobi map
\[
\Phi:X_1(1)\rightarrow \text{Jac}(X_1(1))=0
\]
is constant. In particular, for any $x'\in X_1(1)(\mathbb{Q})$ distinct from $x_1$, one has $\Phi(x_1)=\Phi(x')$. Hence $x_1$ is $\mathbb{P}^1$-parameterized and therefore not isolated. Consequently, $j(E)$ is not an isolated $j$-invariant.
\end{proof}
\begin{remark}
In particular, no isolated points arise from Serre curves. This strengthens the conclusion of \cite{BOURDON2019106824}, Corollary~1.2. We also note that a more general case of this remark has been worked out as an example in \cite{terao2025isolatedpointsmodularcurves}.
\end{remark}

\printbibliography
\end{document}